\documentclass[12pt,reqno]{article}
\usepackage{amsmath, amsthm, graphicx,amsfonts, amssymb,color}
\usepackage{bm}
\usepackage{hyperref}
\newcommand{\lj}{\color{red}}
\newcommand{\ky}{\color{blue}}
\newcommand{\wt}{\color{magenta}}
\newtheorem{thm}{Theorem}[section]
\newtheorem{cor}[thm]{Corollary}
\newtheorem{lem}[thm]{Lemma}

\theoremstyle{definition}

\theoremstyle{remark}
\newtheorem{rem}[thm]{Remark}
\theoremstyle{example}

\numberwithin{equation}{section}
\usepackage{mathrsfs}
\newcommand{\scr}[1]{\mathscr #1}

\newcommand{\eps}{\varepsilon}
\newcommand{\nn}{\nonumber}

 \def\d{\mathrm{d}}

\def\e{\scr E}

\def\bR{\scr R}

\def\cL{\mathcal L}
\def\cM{\mathcal M}
\def\cN{\mathcal N}

\def\bE{\mathbb E}
\def\bP{\mathbb P}

\def\bR{\mathbb R}

\def\bg{\begin}
\def\be{\bg{equation}}
\def\de{\end{equation}}

\def\edar{\end{eqnarray}}

\def\lb{\label}

\def\l{\left}\def\r{\right}
\def\eps{\epsilon}

\def\q{\quad}

\def\lan{\langle}\def\ran{\rangle}

\def\[{\l[} \def\]{\r]}
\def\({\l(} \def\){\r)}

\def\hat{\widehat}
\def\bar{\overline}

 \def\beqlb{\begin{eqnarray}}\def\eeqlb{\end{eqnarray}}
 \def\beqnn{\begin{eqnarray*}}\def\eeqnn{\end{eqnarray*}}
\def\d{{\mbox{\rm d}}}

\title{\bf  {Variational formulas for the exit time of Hunt processes generated by semi-Dirichlet
  forms}}

\author{Lu-Jing Huang\qquad Kyung-Youn Kim\qquad Yong-Hua Mao\qquad Tao Wang\thanks{Corresponding author:  wang\_tao@mail.bnu.edu.cn}}

\date{}

\begin{document}

%%------------------------------------------------------------
 \maketitle

%%------------------------------------------------------------

\bg{abstract}
Variational formulas for the Laplace transform of the exit time from an open set of a Hunt process generated by a regular lower bounded semi-Dirichlet form are established.
While for symmetric Markov processes, variational formulas are derived for the exponential moments of the exit time.
As applications, we provide some comparison theorems and quantitative relations of the exponential moments and Poincar\'e inequalities.

\end{abstract}



%{\bf Running head} Variational principles for the exit time of Hunt processes

%%%%%%%%%%%%%%%%%%%%%%%%%%%%%%%%%%%%%%%%%%%%%%%%%%%%%%%%%%%%%%%%%%%%
\section{Introduction}

The main concern of this paper is to study the variational formulas for the Laplace transform of the exit time of a Hunt process generated by a lower bounded semi-Dirichlet form and for the exponential moment of the exit time of a reversible ergodic Markov process.

The exit time is a fundamental notion in probability.
It plays an important role in the analysis of stochastic processes, especially in ergodic theory, potential theory and martingale theory (\cite{cmf05,IW89,KSK76}).
It arises naturally in various applications too, e.g. in statistical physics, mathematical finance and biology (see \cite{MOR14,Re01} and references therein).

There are classical results for distributions and moments of exit times of some Markov processes (see e.g. \cite{APP05,BS02,DLM17,KM99,KMM98,Rs88}).
The relationship between the exponential moments of the exit time and Poincar\'e inequalities is an important topic in ergodic theory (see e.g. \cite{CGZ13,Fr73,KAM11}).
When the explicit expression of the density of the exit time is not available, its Laplace transform might be very useful (see \cite{Fu80,MR92,Ma97} and the references therein).
For example, analytic expressions for the Laplace transform of the exit time of some special L\'evy processes can be found in e.g. \cite{AAP04,Di07,MP12}.

However, the study of the Laplace transform and exponential moments of the exit times of general Markov processes is more challenging especially for non-symmetric cases due to the lack of tools to deal with non-self-adjoint operators.

Variational formulas are powerful tools. E.g. they provide some nice lower bounds of the spectral gap of Markov processes (\cite{Ch96,CW97}), they are helpful in proof of the recurrence of Markov processes (\cite{AF02,BGL07,GL82,KM99}).

Note that the results cited above are restricted to symmetric Markov processes. Few related results for the non-symmetric case, we refer the reader to \cite{Pi88} for the Dirichlet eigenvalue of diffusions, and \cite{Do94,GL14,LMS19} for examples involving Dirichlet's equations and capacity.
Very recently, started from Poisson's equations, Huang, Kim and Mao extended the variational formulas for the mean exit time to non-symmetric Markov chains and diffusions. In the same work, the authors also obtained some new variational formulas for the Laplace transform and the exponential moments of the exit time (see \cite{HKM20+,HM18} for more details).
Their main idea is to establish some variational formulas for a non-symmetric Markov process relying on its dual process. Indeed, one needs to consider a pair of Poisson's equations together (see Remark \ref{rem-mr1} (1) for more details).

On the other hand, notice that in semi-Dirichlet form theory, a lower bounded semi-Dirichlet form generates a process and its dual process simultaneously.
Furthermore, their $\beta$-potentials, which are closely related to Poisson's equations, have a natural relation to the Dirichlet form and the process (see e.g. \cite[Chapter IV]{MR92} and \cite[Chapter 3]{O13}).
Therefore to extend the variational formulas for the exit time in \cite{HKM20+,HM18} to general Markov processes, the Hunt process generated by a lower bounded semi-Dirichlet form seems to be a good starting point.

Let $X:=(X_t)_{t\geq0}$ be the Hunt process generated by a regular lower bounded semi-Dirichlet form $(\e,\mathscr{F})$ on a locally compact separable metric measure space $(E,d,\mu)$ (see more details in Section \ref{mr1}). We assume that $\mu$ is a positive $\sigma$-finite measure with full support on $E$.
Denote by $\bE_x$ the corresponding expectation starting from $x$ and $\bE_\mu:=\int_E\bE_x\mu(\d x)$.
Let $L^2(E,\mu)$ be the space of square integrable measurable functions on $E$ with respect to $\mu$, furnished with
%%%LJ I drop "its" here
scalar product $\lan f,g\ran_2:=\int_E f(x)g(x)\mu(\d x)$ and the associated norm
$||f||_2:=\sqrt{\lan f,f\ran_2}$.
Define
\be\lb{d:ebeta}
\e_\beta(f,g)= \e(f,g)+\beta\lan f,g\ran_2\q \text{for }\beta>\beta_0,
\de
where $\beta_0$ is the constant in {\bf Assumption A} in Section 2 below.

For any open set $\Omega\subset E$, denote by $\tau_\Omega=\inf\{t\geq0:X_t\notin\Omega\}$
 the first exit time of process $X$ from $\Omega$. We define the function spaces
 $$
 \scr{F}^\Omega=\{f\in\scr{F}:\check{f}=0\ \text{q.e. on }\Omega^c\}
 $$
 and
	$$
	\cN_{\Omega,\delta}=\{f\in\mathscr{F}^\Omega:\mu(f):=\int_E f\d \mu=\delta\}\q \text{for }\delta=0,1,
	$$
where q.e. stands for quasi-everywhere, and $\check{f}$ is the quasi-continuous version of $f$ (see e.g. \cite[p.51]{O13} for more details).

The following is our first main result.
\begin{thm}\lb{main-lt}
	Let $X$ be the Hunt process generated by a regular lower bounded semi-Dirichlet form $(\e,\mathscr{F})$ and $\Omega\subset E$ an open set. Then for any $\beta>\beta_0$,
	\be\lb{vp-lt}
	\frac{\beta}{\bE_\mu[1-\exp(-\beta\tau_\Omega)]}=\inf_{f\in \cN_{\Omega,1}}\sup_{g\in \cN_{\Omega,0}}\e_\beta(f+g,f-g).
	\de
Additionally, if the semi-Dirichlet form $(\e^\Omega,\mathscr{F}^\Omega)$, defined by $\e^\Omega=\e$ on $\mathscr{F}^\Omega$, is transient and $\beta_0=0$, then
	\be\lb{vp-m}
	1/\bE_\mu[\tau_\Omega]=\inf_{f\in \cN_{\Omega,1}}\sup_{g\in \cN_{\Omega,0}}\e(f+g,f-g).
	\de
\end{thm}

\begin{rem}\lb{rem-mr1}
\begin{itemize}

\item[(1)]	Note that the above variational formulas for
the exit time are derived from Poisson's equations \eqref{poi-lt}--\eqref{poi-m} below.
    We point out that we only require the weak solutions of \eqref{poi-lt}--\eqref{poi-m}, which always exist in our setting. Indeed,  they are known as the potentials of $\Omega$ in the classical theory of Dirichlet form (see e.g. \cite{Fu80,O13}).
    Moreover,
    from the proof of Theorem \ref{main-lt} below, one can find that the associated Poisson's equations of its dual process also play an important role.
    Specifically, we construct functions from both of their solutions to attain the inf and sup in \eqref{vp-lt}--\eqref{vp-m}.
    %%%LJ I drop those sentences:  Fortunately, in semi-Dirichlet form theory, $\beta-$potentials of a lower bounded semi-Dirichlet form, which correspond to Poisson's equations of the process and its dual process, have been researched deeply. In particular,
    %they satisfy equality \eqref{solution} below. It is in the line with the idea that we could deal with the Poisson's equations of the process and its dual process together. Motivated by above analysis,  we consider the Hunt processes generated by a lower bounded semi-Dirichlet form and obtain Theorem \ref{main-lt}.

\item[(2)] In our setting,  $\scr{F}^\Omega$ is the $\e_{\beta}$-closure of $C_0(\Omega)\cap\scr{F}$ for all $\beta>\beta_0$ (see \cite[Theorem 3.5.7]{O13}), where $C_0(\Omega)$ is the space of continuous functions with compact support on $\Omega$. Thus we can replace $\cN_{\Omega,\delta}, \delta=0,1$ in \eqref{vp-lt}--\eqref{vp-m} with
    $$
    \cN'_{\Omega,\delta}:=\{f\in C_0(\Omega)\cap\scr{F}:\mu(f)=\delta\},\quad \delta=0,1.
    $$

\item[(3)]	As an application of Theorem \ref{main-lt}, in Subsection \ref{mr1-app} we consider jump diffusion $ X^{(k)}$ with a growing drift whose infinitesimal
    generator is given by
    $$
    \cL_k f(x):=\nabla\cdot a\nabla f(x)-kb\cdot\nabla f(x)+\Delta^{\alpha/2}f(x),\q k\in \bR,\ f\in C_0^\infty(\bR^d).
    $$
    From Theorem \ref{main-lt}, we obtain that the Laplace transform of the exit time of process $X^{(k)}$ and that of its dual process are same. Moreover, the Laplace transform of the exit time of the process, perturbed by a growing drift, is non-decreasing (see Theorem \ref{comp-thm1} below). That is, informally we could say that the exit times of a process and its dual process share the same distribution, and the growing drift reduces the mean time
     that the process exits
       from any open set $\Omega$,
       i.e., intuitively perturbing the growing drift to a process accelerates the convergence rate. Roughly speaking, for any measurable subset $A$,
       $\bP_\mu[X_t\in A, t<\tau_\Omega]\le \bP_\mu[t<\tau_\Omega]\le \bE_\mu[\tau_\Omega]/t$, which gives a faster rate when $\bE_\mu[\tau_\Omega]$ becomes smaller.

    \end{itemize}
\end{rem}

In particular, if the semi-Dirichlet form $(\e,\mathscr{F})$ is symmetric, i.e., $\e(f,g)=\e(g,f)$ for all $f,g\in\mathscr{F}$, the variational formulas in Theorem \ref{main-lt} are reduced to the following simple forms.

\begin{cor}\lb{revm-lt}
Let $X$ be the Hunt process generated by a regular lower bounded symmetric semi-Dirichlet form $(\e,\mathscr{F})$ and $\Omega\subset E$ an open set. Then for any $\beta>\beta_0$,
	 $$
	 \frac{\beta}{\bE_\mu[1-\exp(-\beta\tau_\Omega)]}=\inf_{f\in \cN_{\Omega,1}}\e_\beta(f,f).
	 $$
Additionally, if the
%%part
semi-Dirichlet form $(\e^\Omega, \mathscr{F}^\Omega)$ is transient and $\beta_0=0$, then
$$
1/\bE_\mu[\tau_\Omega]=\inf_{f\in \cN_{\Omega,1}}\e(f,f).
$$
\end{cor}

Our second main result involves the study of reversible ergodic Markov processes. Let $Y:=(Y_t)_{t\geq0}$ be a right-continuous strong Markov process on a polish space $(E,\mathscr{B}(E))$, with transition kernel $P_t(x,\d y),t\geq0,x,y\in E$.
Assume that process  $Y$ is ergodic, i.e., there exists a unique invariant probability measure $\pi$ such that
$$
\lim_{t\rightarrow\infty}||P_t(x,\cdot)-\pi(\cdot)\|_{\text{var}}=0,\q x\in E,
$$
where $||\nu||_{\text{var}}:=\sup_{|h|\leq 1}|\nu(h)|$ is the total variation norm of a signed measure $\nu$ on $\mathscr{B}(E)$.
We assume that $Y$ is reversible  with respect to $\pi$, that is,
$$
\pi(\d x)P_t(x,\d y)=\pi(\d y)P_t(y,\d x)\q \text{for all }t\geq0,\ x,y\in E.
$$

Let $(L,\scr{D}(L))$ be the infinitesimal generator of process $Y$ in $L^2(E,\pi)$ and denote the associated Dirichlet form by $(\e_\pi,\mathscr{D}(\e_\pi))$. We also assume that $(\e_\pi,\mathscr{D}(\e_\pi))$ is regular, i.e., $\scr{D}(\e_\pi)\cap C_0(E)$ is dense both in $\scr{D}(\e_\pi)$ and $C_0(E)$.
For any $\beta\in\bR$, define
$$
\e_{\pi,\beta}(f,g)=\e_\pi(f,g)+\beta\lan f,g\ran_{2,\pi},
$$
where $\lan\cdot,\cdot\ran_{2,\pi}$ is the usual product in $L^2(E,\pi)$.
To simplify the notation, we also denote the associated exit time by $\tau_\cdot$ and expectation by $\bE_\cdot$.
\begin{thm}\lb{main-exp}
Let $Y$ be a reversible ergodic right-continuous strong Markov process on $(E,\mathscr{B}(E))$ with regular Dirichlet form $(\e_\pi,\mathscr{D}(\e_\pi))$  and $\Omega\subset E$ an open set. Then for any $\beta>0$,
\be\lb{vp-exp}
\frac{\beta}{ \bE_\pi[\exp(\beta\tau_{\Omega})]-1}=\inf_{f\in\cN^\pi_{\Omega,1}}\e_{\pi,-\beta}(f,f) \vee 0,
\de
where
$\cN^\pi_{\Omega,1}:=\{f\in \scr{D}(\e_\pi):\ \check{f}=0\ \text{q.e. on $\Omega^c$ and }\pi(f)=1\}$.
\end{thm}

\begin{rem}\lb{rem-kill}
It is well known that the exponential moments of the exit time are closed related to the local Poincar\'e inequality.
Specifically,
we say that process $Y$ possesses the local Poincar\'e inequality on $\Omega$ if
\begin{equation}\label{poinc-l}
\pi(f^2)\leq c_1^{-1} \e_\pi(f,f) \q \text{for all}\ f\in \scr{D}(\e_\pi)\ \text{with}\ \check{f}=0\ \text{q.e. on }\Omega^c
\end{equation}
for some $c_1>0$.
Denote by $\lambda_0(\Omega)$ its optimal constant, which is given by
\be\lb{di-ei}
\lambda_0(\Omega)=\inf\{\e_\pi(f,f):\ f\in\mathscr{D}(\e_\pi), \ \check{f}=0\ \text{q.e. on }\Omega^c\ \text{and }\pi(f^2)=1\}.
\de
Then
$$
\lambda_0(\Omega)=\sup\{\beta>0:\bE_\pi[\exp(\beta\tau_\Omega)]<\infty\},
$$
we refer to \cite{Fr73,LLL14} for more details.
As we shall see,  when $\beta<\lambda_0(\Omega)$, \eqref{vp-exp} is reduced to
$$
\frac{\beta}{\bE_\pi[\exp(\beta\tau_{\Omega})]-1}=\inf_{f\in\cN^\pi_{\Omega,1}}\e_{\pi,-\beta}(f,f).
$$

%We say that process $Y$ possesses the local Poincar\'e inequality on $\Omega$ if
%\begin{equation}\label{poinc-l}
%\pi(f^2)\leq c_1^{-1} \e_\pi(f,f), \q \text{for all}\ f\in \scr{D}(\e_\pi)\ \text{with}\ \check{f}=0\ \text{q.e. on }\Omega^c.
%\end{equation}
%for some $c_1>0$. Let $\lambda_0(\Omega)$ be its optimal constant, which is given by
%\be\lb{di-ei}
%\lambda_0(\Omega)=\inf\{\e_\pi(f,f):\ f\in\mathscr{D}(\e_\pi), \ \check{f}=0\ \text{q.e. on }\Omega^c\ \text{and }\pi(f^2)=1\}.
%\de
%Since $\scr{D}(\scr{E}_\pi^\Omega):=\{f\in\mathscr{D}(\e_\pi), \check{f}=0\ \text{q.e. on }\Omega^c\}$ is the $\e_{\pi,1}$-closure of $C_0(\Omega)\cap\scr{D}(\scr{E}_\pi),$ we have
%$ \lambda_0(\Omega)=\inf\{\e_\pi(f,f):\ f\in\mathscr{D}(\e_\pi), \ f=0\ \pi\text{-a.e. on }\Omega^c\ \text{and }\pi(f^2)=1\}.$
%It is well known that the exponential moments of the exit time are closed related to the local Poincar\'e inequality.
%Indeed,
%$$
%\lambda_0(\Omega)=\sup\{\beta>0:\bE_\pi[\exp(\beta\tau_\Omega)]<\infty\},
%$$
%we refer to \cite{Fr73,LLL14} for more details.
%As we shall see,  when $\beta<\lambda_0(\Omega)$, \eqref{vp-exp} is reduced to
%$$
%\frac{\beta}{\bE_\pi[\exp(\beta\tau_{\Omega})]-1}=\inf_{f\in\cN^\pi_{\Omega,1}}\e_{\pi,-\beta}(f,f).
%$$
\end{rem}

Set
\be\lb{pi-c}
\lambda_1=\inf\{\e_\pi(f,f):\ f\in \scr{D}(\e_\pi),\pi(f)=0\ \text{and }\pi(f^2)=1\},
\de
 which is the optimal constant of Poincar\'e inequality(see e.g. \cite[Example 1.1.2]{wfy05}).
Now as an application of Theorem \ref{main-exp}, some quantitative inequalities of the exponential moments of the exit time, $\lambda_0(\Omega)$ and $\lambda_1$ are presented as follows.

\begin{cor}\lb{up-low-b}
Let $Y$ be the process  as in Theorem \ref{main-exp}.
Let $ \Omega\subset E$ be an open set and assume that $\lambda_0(\Omega)>0$.  Then the following statements hold.
\begin{itemize}
\item[\rm(1)] {\bf(Upper bounds for exponential moments)} For any $\beta\in (0, \lambda_0(\Omega))$,
\be\lb{rep-expint}
\bE_\pi[\exp(\beta\tau_{\Omega})]\leq 1+\frac{\beta}{\lambda_0(\Omega)-\beta}.
\de
 In particular, if  $\lambda_1>0$ and $\pi(\Omega^c)>0$, then for any $\beta\in (0, \lambda_1\pi(\Omega^c))$,
\be\lb{exp-lambda1}
\bE_\pi[\exp(\beta\tau_{\Omega})]\leq 1+\frac{\beta}{\lambda_1\pi(\Omega^c)-\beta}.
\de

\item[\rm(2)] {\bf (Lower bounds for exponential moments)} Additionally, if
\be\lb{dirichlet-e}
 \begin{cases}
 (-\lambda_0(\Omega)-L)u=0,&\ \text{in }\Omega;\\
 u=0,&\ \text{q.e. on }\Omega^c
 \end{cases}
 \de
has a weak solution $\phi\in\scr{D}(\e_\pi)$, i.e., $\phi=0$ q.e. on $\Omega^c$ and
$$
\e_{\pi,-\lambda_0(\Omega)}(\phi,f)=0\q \text{for all }f\in\scr{D}(\e_\pi)\ \text{with }\check{f}=0\ \text{q.e. }\Omega^c,
$$
then for any $\beta\in (0, \lambda_0(\Omega))$,
\be\lb{expint-lb}
    \bE_\pi[\exp(\beta\tau_\Omega)]\geq 1+\frac{\beta\pi(\phi)^2}{(\lambda_0(\Omega)-\beta)\pi(\phi^2)}.
\de
\end{itemize}
\end{cor}
\begin{rem}
\begin{itemize}

\item[(1)] We mention that the existence assumption for solution of \eqref{dirichlet-e} is not too strong. In fact, many cases satisfy this condition. For example, it follows from \cite[Theorem 3.5.5]{Pi95} that \eqref{dirichlet-e} of diffusions with some regularity conditions has a unique positive strong solution.

\item[(2)] The similar bounds for the Laplace transform of the exit time and the mean exit time are also obtained from Theorem \ref{main-lt}, which will be presented in Subsection \ref{exp-poinc}.
\end{itemize}
\end{rem}

The remaining part of this paper is organized as follows.
In Subsection \ref{mr1-proof} the proofs of Theorem \ref{main-lt} and Corollary \ref{revm-lt} are presented. We introduce diffusions with $\alpha$-stable jumps in order to illustrate the application of Theorem \ref{main-lt} in Subsection \ref{mr1-app}.
Subsection \ref{proof-mr2} is devoted to the proof of Theorem \ref{main-exp}, while Subsection \ref{exp-poinc} provides the proof of Corollary \ref{up-low-b} and some bounds of the Laplace transform of the exit time and the mean exit time. Finally, a reversible ergodic diffusion on manifold is provided as an example in Subsection \ref{exm-sym}.

%%%%%%%%%%%%%%%%%%%%%%%%%%%%%%%%%%%%%%%%%%%%%%%%%%%%%%%%%%%%%%%%%%%%%%%%%%%%%%%%%%%%%%%
\section{The exit time of Hunt processes}\lb{mr1}

The main objective of this section is to prove Theorem \ref{main-lt} and give an example to illustrate its application. For that, we present a few definitions that will be used later.

Recall that $(E, d, \mu)$ is the locally compact separable metric measure space equipped with metric $d$ and positive $\sigma$-finite
measure $\mu$ with full support.
Let $\mathscr{F}$ be a dense subspace of $L^2(E,\mu)$. A bilinear form $\e$ defined on $\mathscr{F}\times \mathscr{F}$ is called {\it a lower bounded semi-Dirichlet form} if the following {\bf Assumption A} is satisfied: there exists a constant $\beta_0\geq0$ such that
\begin{itemize}
\item[(1)] (lower boundedness) for any $f\in\mathscr{F}$, $\e_{\beta_0}(f,f)\geq0$, where $\e_\beta(f,g):=\e(f,g)+\beta\langle f,g\rangle_2$.

\item[(2)] (weak sector condition) there exists a constant $C\geq1$ such that
$$
|\e(f,g)|\leq C\sqrt{\e_{\beta_0}(f,f)}\sqrt{\e_{\beta_0}(g,g)}\q \text{for all }f,g\in \mathscr{F}.
$$
\item[(3)] (closedness)
 $\mathscr{F}$ is a Hilbert space relative to the inner product
 $$
 \frac{1}{2}\big(\e_\beta(f,g)+\e_\beta(g,f)\big)\quad \text{for all }\beta>\beta_0.
 $$
\item[(4)] (Markov property) $f^+\wedge1\in \mathscr{F}$ whenever $f\in \mathscr{F}$, and $\e(f^+\wedge1,f-f^+\wedge1)\geq0$.
\end{itemize}
We say that a lower bounded semi-Dirichlet form $(\e,\mathscr{F})$ is regular if $\mathscr{F}\cap C_0(E)$ is uniformly dense in $C_0(E)$ and $\e_\beta$-dense in $\mathscr{F}$ for $\beta>\beta_0$.

It is well known that there exists a Hunt process $X:=(X_t)_{t\geq0}$ associated with a regular lower bounded semi-Dirichlet form $(\e,\mathscr{F})$ (see e.g. \cite[Theorem 3.3.4]{O13}).
Corresponding to process $X$, there exists a unique semigroup  $(P_t)_{t\geq0}$ on $L^2(E,\mu)$.
Denote the associated resolvent by $R_\beta=\int_0^\infty e^{-\beta t}P_t dt$, $\beta>0$.
From \cite[Theorem 1.1.2]{O13}, we
 see that there also exists a strongly continuous semigroup $(\widetilde{P}_t)_{t\geq0}$ such that
$$
\lan P_t f,g\ran_2=\lan f, \widetilde{P}_t g\ran_2, \q t\geq0,\ f,g\in L^2(E,\mu).
$$
That is, $\widetilde{P}_t$ is the dual operator of $P_t$ with respect to $\mu$. Furthermore, let $\widetilde{R}_\beta=\int_0^\infty e^{-\beta t}\widetilde{P}_t dt,\ \beta>0$ be the resolvent of semigroup $(\widetilde{P}_t)_{t\geq0}$, which is also called the dual resolvent of $(R_\beta)_{\beta>0}$. Then for $\beta>\beta_0,\ f\in L^2(E,\mu)$ and $h\in\mathscr{F}$,
$$
\e_\beta(R_\beta f,h)=\e_\beta(h,\widetilde{R}_\beta f)=\lan f,h\ran_2.
$$
%and the adjoint operator $(\widetilde P_t)_{t\geq0}$ is defined as
%$\lan P_t f,g\ran_2=\lan f,\widetilde{P}_t g\ran_2$. For $\beta>\beta_0$, denote the associated resolvents by
%$$
%R_\beta:=\int_0^\infty e^{-\beta t}P_t dt\q\text{and}\q \widetilde{R}_\beta:=\int_0^\infty e^{-\beta t}\widetilde{P}_t dt,
%$$
%which satisfy that
%$$
%\e_\beta(R_\beta f,h)=\e_\beta(h,\widetilde{R}_\beta f)=\lan f,h\ran_2
%$$
%for $f\in L^2(E,\mu),\ h\in\mathscr{F}$.
%We also call that $\widetilde{P}_t, \widetilde{R}_\beta$ are the dual operators of $P_t$ and $R_t$ in $L^2(E,\mu)$ respectively.
In addition, (A4) in {\bf Assumption A} yields that $(R_\beta)_{\beta>0}$ can be extended to  a sub-Markov resolvent on $L^\infty(X, \mu)$.
Now define the potential operator
\be\lb{ivs}
Rf=\lim_{n\rightarrow\infty}R_{1/n}f,\q f\in L^\infty(X, \mu).
\de
We say that the semi-Dirichlet form $(\e,\mathscr{F})$ is transient if there exists a strictly positive function $f\in L^\infty(E,\mu)$ such that $Rf<\infty, \mu$-a.e. (see, e.g. \cite[p.13]{O13}).

%Fix a domain $\Omega\subset E$ and function $\xi\in L^2(E,\mu)$ with $\xi|_{\Omega^c}=0$.
%From \cite[(4.1.1) and Example 4.1.3]{O13} we could see that
%\be\lb{rsv-2}
%R_\beta \xi(x)=R_\beta (\xi1_\Omega)(x)=\int_0^\infty \exp(-\beta t)P_t\xi(x)\d t,\q \beta>\beta_0.
%\de

\subsection{Proof of Theorem \ref{main-lt}}\lb{mr1-proof}
Recall that $X=(X_t)_{t\geq0}$ is the process generated by a regular lower bounded semi-Dirichlet form $(\e,\mathscr{F})$. Denote $(P_t)_{t\geq0}$ and $(R_\beta)_{\beta\geq0}$ by its associated transition semigroup and resolvent respectively. Let $(\cL,\scr{D}(\cL))$ be its infinitesimal generator in $L^2(E,\mu)$, that is,
$$
\aligned
\scr{D}(\cL)&:=\bigg\{f\in L^2(E,\mu) : \lim_{t\to 0}\frac{P_t f-f}{t} \ \text{exists in $L^2(E,\mu)$}\bigg\},\\
\cL f(x) &:=\lim_{t\to 0}\frac{P_t f(x) -f(x)}{t}.
\endaligned
$$

%As we mentioned before, we will derive the variational formulas \eqref{vp-lt}--\eqref{vp-m} from Poisson's equations.
To prove Theorem \ref{main-lt}, we introduce a class of Poisson's equations as follows.
Fix an open set $\Omega\subset E$, a function $\xi\in L^2(\Omega,\mu|_\Omega)$
and $\beta>\beta_0$, where $\beta_0$ is the constant in {\bf Assumption A}.
Consider the Poisson's equation corresponding to $\cL$:
\be\lb{poi}
\begin{cases}
(\beta-\cL)u=\xi ,&\q \text{in }\Omega;\\
u=0,&\q \text{q.e. on }\Omega^c.
\end{cases}
\de
The function $u_\beta\in \mathscr{F}^\Omega$ is called a weak solution of \eqref{poi} if
$$
\e_\beta(u_\beta, f)=\lan \xi,f\ran_2:=\int_\Omega\xi f\d\mu\q \text{for all }f\in\mathscr{F}^\Omega.
$$
In fact, from the proof of \cite[Theorem 3.5.7]{O13}, one could find that
$$
R^\Omega_\beta\xi(x):=\bE_x\int_0^{\tau_\Omega}\text{e}^{-\beta t}\xi(X_t)\d t,\q x\in E
$$
is the unique weak solution of \eqref{poi}.
Furthermore,
\be\lb{solution}
\e_\beta(R_\beta^\Omega \xi,f)=\e_\beta(f,\widetilde{R}^\Omega_\beta \xi)=\lan\xi,f\ran_2\q\text{for all } f\in \mathscr{F}^\Omega,
\de
where $\widetilde{R}^\Omega_\beta$ is the dual resolvent of $R^\Omega_\beta$.
%and similarly it is defined using the corresponding exit time $\widetilde{\tau}_\Omega$.

Denote two function spaces $\cM_{\Omega,\delta}$ by
$$
\cM_{\Omega,\delta}=\{f\in\mathscr{F}^\Omega:\lan\xi,f\ran_2=\delta\},\q \delta=0,1.
$$
For \eqref{poi}, we obtain the following variational formulas.

\begin{thm}\lb{main-dp}
Let $X$ be the process as in Theorem \ref{main-lt} and  $\Omega\subset E$  an open set, $\beta>\beta_0$ and function $\xi\in L^2(\Omega,\mu|_\Omega)$.
Denote by $u_\beta=R_\beta^\Omega\xi \in \mathscr{F}^\Omega$ the  unique weak solution of  \eqref{poi}.
Then
	\be\lb{xi-vp}
	1/\lan\xi,u_\beta\ran_2= \inf_{f\in \cM_{\Omega,1}}\sup_{g\in { \cM_{\Omega,0}}}\e_\beta(f+g,f-g).
	\de
In particular, if $\e$ is symmetric, then
	\be\lb{revxi-vp}
	1/\lan\xi,u_\beta\ran_2= \inf_{f\in \cM_{\Omega,1}}\e_\beta(f,f).
	\de
\end{thm}

\begin{proof}
Fix $\beta>\beta_0$.
Denote by $\widetilde{u}_\beta=\widetilde{R}_\beta^\Omega\xi\in \mathscr{F}^\Omega$.
Then it follows from \eqref{solution} that
\be \lb{xi-u}
\lan\xi,\widetilde{u}_\beta\ran_2=\e_\beta(\widetilde{u}_\beta,\widetilde{u}_\beta) =\e_\beta(u_\beta,\widetilde{u}_\beta)=\e_\beta(u_\beta,u_\beta)=\lan\xi,u_\beta\ran_2.
\de
For the convenience of the notation, we set
	$$
	w_\beta=\frac{u_\beta}{\lan\xi, u_\beta\ran_2}\ \text{ and }\ \widetilde{w}_{\beta}=\frac{\widetilde u_\beta}{\lan\xi, u_\beta\ran_2}.
	$$
Then from \eqref{xi-u} and Appendix,  it is easy to check that
	$$
	\bar{w}_\beta:=(w_\beta+\widetilde{w}_\beta)/2\in \cM_{\Omega,1}\q \text{and}\q \hat{w}_\beta:=(w_\beta-\widetilde{w}_\beta)/2\in \cM_{\Omega,0}.
	$$
Moreover,
	\be\lb{n-w-w}
	\e_\beta(w_\beta,w_\beta)=\e_\beta(w_\beta,\widetilde{w}_\beta) =1/\lan\xi,u_\beta\ran_2=1/\e_\beta(u_\beta,\widetilde{u}_\beta).
	\de
	
To prove \eqref{xi-vp}, we first take $f=\bar{w}_\beta$ in the infimum. For any $g\in\cM_{\Omega,0}$, let $g_1=g-\hat{w}_\beta$. It is clear that $g_1\in\cM_{\Omega,0}$ since $\hat{w}_\beta\in \cM_{\Omega,0}$. Combining this fact with \eqref{solution} and \eqref{xi-u}, we get
	\begin{equation}\lb{n-geq1}
	\e_\beta(w_\beta,g_1)=\e_\beta(g_1,\widetilde{w}_\beta)=\frac{\lan \xi,g_1\ran_2}{\e_\beta(u_\beta,\widetilde{u}_\beta)}=0.
	\end{equation}
	Thanks to \eqref{n-w-w}--\eqref{n-geq1} and the fact $\e_\beta(g_1,g_1)\geq 0$ by {\bf Assumption A} (A1), we have
	$$
	\aligned
	\e_\beta(\bar{w}_\beta+g,\bar{w}_\beta-g)&=\e_\beta(\bar{w}_\beta+\hat{w}_\beta+g_1,\bar{w}_\beta-\hat{w}_\beta-g_1)\\
	&=\e_\beta(w_\beta,\widetilde{w}_\beta)+\e_\beta(g_1,\widetilde{w}_\beta)-\e_\beta(w_\beta,g_1) -\e_\beta(g_1,g_1)\leq 1/\e_\beta(u_\beta,\widetilde{u}_\beta),
	\endaligned
	$$
  this implies
	\be\lb{n-geq}
	1/\e_\beta(u_\beta,\widetilde{u}_\beta)\geq \inf_{f\in \cM_{\Omega,1}}\sup_{g\in { \cM_{\Omega,0}}}\e_\beta(f+g,f-g).
	\de

On the other hand, we take $g=\hat{w}_\beta$ in the supremum of \eqref{xi-vp}. For any $f\in \cM_{\Omega,1}$, set $f_1=f-\bar{w}_\beta$. We also have that $f_1\in\cM_{\Omega,0}$ from $\bar{w}_\beta\in \cM_{\Omega,1}$. With replacing $g_1$ by $f_1$ in \eqref{n-geq1}, we arrive at
	$$
	\aligned	\e_\beta(f+\hat{w}_\beta,f-\hat{w}_\beta)&=\e_\beta(w_\beta+f_1,\widetilde{w}_\beta+f_1)\\
	&=\e_\beta(w_\beta,\widetilde{w}_\beta)+\e_\beta(w_\beta,f_1)+\e_\beta(f_1,\widetilde{w}_\beta)+\e_\beta(f_1,f_1)\geq 1/\e_\beta(u_\beta,\widetilde{u}_\beta),
	\endaligned
	$$
	which implies
	\be\lb{n-leq}
	1/\e_\beta(u_\beta,\widetilde{u}_\beta)\leq \inf_{f\in \cM_{\Omega,1}}\sup_{g\in \cM_{\Omega,0}}\e_\beta(f+g,f-g).
	\de
	Combining \eqref{n-geq}--\eqref{n-leq}, we obtain the first assertion.
In particular, if $\e$ is symmetric,  from the fact
	$$\e_\beta(f+g,f-g)= \e_\beta(f,f)-\e_\beta(g,g)\leq \e_\beta(f,f),$$
the proof is complete immediately.
\end{proof}

\begin{rem}
 As we have been seen that the key point in the proof of Theorem \ref{main-dp} is the existence of the solution to \eqref{poi}, which is widely studied in probability, see for example \cite{Fu80,MR92,O13} in Dirichlet theory and \cite{Pi95} in diffusions. It is known that \eqref{poi} is closed related to the properties of the associated process, such as the convergence rate of the semigroup (see e.g. \cite{Ma97}), we leave a more detailed discussion starting from Theorem \ref{main-dp} in this topic for future work.
\end{rem}

We now proceed to prove Theorem \ref{main-lt} and Corollary \ref{revm-lt}.

\medskip
\noindent{\bf Proof of Theorem \ref{main-lt}.}
Fix an open set $\Omega\subset E$ and $\beta>\beta_0$. We first suppose  that $\mu(\Omega)<\infty$. Since $\beta R_\beta^\Omega {\bf 1}_{\Omega}(x)=1-\bE_x[\exp(-\beta\tau_\Omega)]$, \eqref{solution} yields that $u_\beta:=\big((1-\bE_x[\exp(-\beta\tau_\Omega)])/\beta:x\in E\big)$ is the unique weak solution of Poisson's equation
\be\lb{poi-lt}
\begin{cases}
	(\beta-\cL)u=1,&\q \text{in }\Omega;\\
	u=0,&\q \text{q.e. on }\Omega.
\end{cases}
\de
Applying \eqref{xi-vp} to \eqref{poi-lt}, we obtain \eqref{vp-lt}.

Assume additionally that $(\e^\Omega,\mathscr{F}^\Omega)$ is transient and $\beta_0=0$. According to \cite[Theorem 1.3.9]{O13} and $R^\Omega{\bf 1}(x)=\bE_x[\tau_\Omega]$, where operator $R^\Omega$ is defined in \eqref{ivs} replacing $R_{1/n}$ with $R^\Omega_{1/n}$, we see that $u_0:=(\bE_x[\tau_\Omega]:x\in E)$ is the weak solution to
\be\lb{poi-m}
\begin{cases}
	-\cL u=1,&\q \text{in }\Omega;\\
	u=0,&\q \text{q.e. on }\Omega^c.
\end{cases}
\de
Therefore, \eqref{vp-m} follows from \eqref{xi-vp} and \eqref{poi-m}.

For general open set $\Omega$, we only prove \eqref{vp-lt}, and the proof of \eqref{vp-m} is quite similar.
Let $(\Omega_n)_{n\geq1}$ be a sequence of open sets such that $\mu(\Omega_n)<\infty$, $\Omega_n\subset\Omega_{n+1}$ and $\cup_{n=1}^\infty\Omega_n=\Omega$. Denote
$$
\cN'_{\Omega_n,\delta}=\{f\in C_0(\Omega_n)\cap\scr{F}:\mu(f)=\delta\}\ \text{and}\ \cN'_{\Omega,\delta}=\{f\in C_0(\Omega)\cap\scr{F}:\mu(f)=\delta\},\q \delta=0,1.
$$
We claim that
\be\lb{vp-lt'}
\frac{\beta}{\bE_\mu[1-\exp(-\beta\tau_\Omega)]}
=\inf_{f\in\cN_{\Omega,1}'}\sup_{g\in\cN_{\Omega,0}'}\e_\beta(f-g,f+g),
\de
and then \eqref{vp-lt} holds as we mentioned in Remark \ref{rem-mr1} (2).
Indeed, for any $f_0\in\cN_{\Omega,1}'$, there is integer $m>0$ such that $\text{supp}(f_0)\subset \Omega_m$, thus $f_0\in\cN_{\Omega_m,1}'$. Applying the above argument and Remark \ref{rem-mr1} (2) to $\Omega_m$, we arrive at
$$
\aligned
\frac{\beta}{\bE_\mu[1-\exp(-\beta\tau_{\Omega_m})]}&=\inf_{f\in \cN_{\Omega_m,1}'}\sup_{g\in\cN_{\Omega_m,0}'}\e_\beta(f-g,f+g)\leq \sup_{g\in\cN_{\Omega_m,0}'}\e_\beta(f_0-g,f_0+g)\\
&\leq \sup_{g\in\cN_{\Omega,0}'}\e_\beta(f_0-g,f_0+g).
\endaligned
$$
Combining this with the fact $\tau_{\Omega_m}\leq \tau_\Omega$ gives that
$$
\frac{\beta}{\bE_\mu[1-\exp(-\beta\tau_\Omega)]}\leq \inf_{f\in\cN_{\Omega,1}'}\sup_{g\in\cN_{\Omega,0}'}\e_\beta(f-g,f+g).
$$

Moreover, it follows directly from the similar argument in the proof of \cite[Theorem 3.5]{HKM20+} that
$$
\frac{\beta}{\bE_\mu[1-\exp(-\beta\tau_\Omega)]}\geq \inf_{f\in\cN_{\Omega,1}'}\sup_{g\in\cN_{\Omega,0}'}\e_\beta(f-g,f+g).
$$
Hence, by the above analysis \eqref{vp-lt'} holds.     \qed

\medskip
\noindent{\bf Proof of Corollary \ref{revm-lt}.} The assertions follow immediately from Theorem \ref{main-lt} and the symmetry of $(\e,\scr{F})$.
\qed

%%%%%%%%%%%%%%%%%%%%%%%%%%%%%%%%%%%%%%%%%%%%%%%%%%%%%%%%%%%%%%%%%%%
\subsection{An example of diffusions with $\alpha$-stable jumps}\lb{mr1-app}

In this subsection, we provide an explicit example to illustrate Theorem \ref{main-lt}.  Take $E=\bR^d$ and $\mu(\d x)=\d x$ the Lebesgue measure in the rest of this section.

Consider the following integro-differential operator in divergence form on $\bR^d$:
\be\lb{int-dif}
\cL f(x):=\nabla\cdot a\nabla f(x)-b\cdot \nabla f(x)+\Delta^{\alpha/2} f(x),\q f\in C^\infty_0(\bR^d),
\de
where $\Delta^{\alpha/2}:=-(-\Delta)^{\alpha/2}$, $\alpha\in(0,2)$ is the fractional Laplace operator. That is,
\begin{align}
\label{frac}
\Delta^{\alpha/2}f(x)&=\int_{\bR^d}\big(f(x+z)-f(x)-\nabla f(x)\cdot z{\bf1}_{\{|z|\leq1\}}\big)\frac{c_{d,\alpha}}{|z|^{d+\alpha}}dz,
\end{align}
where  $c_{d, \alpha}={\alpha2^{\alpha-1}\Gamma(\frac{ \alpha+d}{2})}/(\pi^{d/2}\Gamma(1-\alpha/2)).$ For coefficients $a$ and $b$, we give the following conditions:
\begin{itemize}
\item[(B1)] $a_{ij}\in C^1(\bR^d),\ 1\leq i,j\leq d$ and there exist constants $0<\Lambda_1\leq \Lambda_2$ such that
$$
\Lambda_1|v|^2\leq v\cdot a(x)v\leq \Lambda_2 |v|^2\q \text{for all }x, v\in\bR^d.
$$

 \item[(B2)] %$b_i\in L^{p_0}(\Omega)$ for some $p_0$ with $d\leq p_0\leq\infty$ if $\Omega$ is bounded and when $\Omega$ is unbounded
    $b_i\in L^d(\bR^d,\d x)$ for $i=1,\cdots,d$ such that
     $
     ||b||_d:=\sum_{i=1}^d \|b_i\|_d\le \frac{\Lambda_1}{2C^*},
     $
    where $C^*$ is the constant in the Gagliado-Nirenberg-Sobolev inequality (e.g. see, \cite[Lemma 3.1]{Ue14}).
\end{itemize}
Define the associated bilinear form by
\begin{align}\label{se_D:3}
\e( f,g)&= \int_{\bR^d} \nabla f(x)\cdot a(x)\nabla g(x)\d x+\int_{\bR^d} f(x)b(x)\cdot \nabla g(x)\d x\\
&\q+\int\int_{x\neq y}\big(f(x)-f(y)\big)\big(g(x)-g(y)\big)\frac{c_{d,\alpha}}{|x-y|^{d+\alpha}}\d x\d y,\q f,g\in C^1_0(\bR^d).\nn
\end{align}

Under conditions (B1)--(B2), \cite[Theorem 3.1]{Ue14} tells us that the bilinear form $\e$ defined by \eqref{se_D:3} extends from $C^1_0(\bR^d)\times C^1_0(\bR^d)$ to $\mathscr{F}\times \mathscr{F}$, and it is a lower bounded closed form on $L^2(\bR^d,dx)$. Moreover, $(\e,\mathscr{F})$ is regular on $L^2(\bR^d,dx)$ so that there is an associated Hunt process $X$. Therefore, Theorem \ref{main-lt} holds for this example.

%Therefore, \eqref{vp-lt} in Theorem \ref{main-lt} holds for this example. Furthermore, for an open set $\Omega\subset \bR^d$, if the killed process $X^\Omega$ is transient (this condition holds, for example, if $X$ is transient or $\mu(\Omega^c)>0$ ) and $\beta_0=0$,  \eqref{vp-m} also holds.

In the following, we give some further observations of the exit time of the Hunt process $X$ by Theorem \ref{main-lt}.
%if $\Omega$ is a bounded domain in $\bR^d$ as follows.
For that, we perturb $\cL$ by a growing drift and define
$$
\cL_k=\nabla\cdot a\nabla -kb\cdot \nabla +\Delta^{\alpha/2},\quad k\in \bR.
$$
When conditions (B1) and (B2) for $kb$ hold, denote $X^{(k)}$ by the associated Hunt process of $\cL_k$. Note that $X^{(0)}$ is a symmetric process with respect to the Lebesgue measure.
For any open set $\Omega\subset\bR^d$, let $\tau^{(k)}_\Omega$ be the exit time of process $X^{(k)}$.

From Theorem \ref{main-lt}, we present the following comparison theorem for the processes $X^{(k)},k\in\bR$.

\begin{thm}\lb{comp-thm1}
Assume that $\cL$ defined by
 \eqref{int-dif} satisfies
  $\text{div}(b):=\sum_{i=1}^d\partial_{x_i} b_i=0$
 and {\rm (B1)--(B2)}. Let $\Omega\subset\bR^d$ be a bounded open set and $k_0=\Lambda_1/(2C^*||b||_d)$. Then the following statements hold.
\begin{itemize}
\item[\rm(1)] For all $|k|\leq k_0$ and $\beta>0$,
\be\lb{lt_k-k}
\int_\Omega\bE_x[\exp(-\beta\tau_\Omega^{(k)})]\d x=\int_\Omega\bE_x[\exp(-\beta\tau_\Omega^{(-k)})]\d x.
\de
Moreover, $\int_\Omega\bE_x[\tau_\Omega^{(k)}] \d x= \int_\Omega\bE_x[\tau_\Omega^{(-k)}]\d x$.

\item[\rm(2)] Fix $\beta>0$.
$\int_\Omega\bE_x[\exp(-\beta\tau_\Omega^{(k)})]\d x$ is non-decreasing and $\int_\Omega\bE_x[\tau_\Omega^{(k)}] \d x$ is non-increasing for $k\in [0,k_0]$.
\end{itemize}
\end{thm}
\begin{proof}
It is straightforward to check that the constant $\beta_0=0$ by \cite[Proposition 3.2]{Ue14} under the conditions (B1)--(B2) and the definition of $\cL$ in \eqref{int-dif}.

For $|k|\le k_0$, let $\e^{(k)}$ be the semi-Dirichlet form
of $\cL_k$ defined by
\begin{align*}
		%\label{se_D:4}
\e^{(k)}( f,g)&= \int_{\bR^d} \nabla f(x)\cdot a(x)\nabla g(x)\d x+k\int_{\bR^d} f(x)b(x)\cdot \nabla g(x)\d x\\
	&\q+\int\int_{x\neq y}\big(f(x)-f(y)\big)\big(g(x)-g(y)\big)\frac{c_{d,\alpha}}{|x-y|^{d+\alpha}}\d x\d y,\q f,g\in C^1_0(\bR^d),\nn
\end{align*}
 %\eqref{se_D:3}
 and $\e^{(k)}_\beta(f,g):=\e^{(k)}(f,g)+\beta\lan f,g\ran_2$, $\beta>0$. For convenience, we also define
$$
\check{\e}(f,g)=\int_{\bR^d}f(x)b(x)\cdot \nabla g(x)\d x,\q f,g\in C_0^1(\bR^d).
$$
It is easy to see that $\check{\e}(f,g)=-\check{\e}(g,f)$ as well as $\e^{(k)}_\beta(f,g)=\e^{(0)}_\beta(f,g)+ k\check{\e}(f,g),\ f,g\in C_0^1(\bR^d).$

Let $\Omega\subset \bR^d$ be a bounded open set. It is clear that $C^1_0(\Omega)$ is $\e_\beta$-dense in $\scr{F}^\Omega$ for any $\beta>0$, thus Theorem \ref{main-lt} is still true if we replace $\mathscr{F}^\Omega$ with $C^1_0(\Omega)$ in the definition of $\cN_{\Omega,\delta},\delta=0,1$.
Therefore, by \eqref{vp-lt} and $\text{div}(b)=0$ we obtain that
\begin{align}
&\frac{\beta}{\int_\Omega (1-\bE_x[\exp(-\beta\tau_\Omega^{(k)})])\d x}\nn\\
&=\inf_{f\in\cN_{\Omega,1}}\sup_{g\in \cN_{\Omega,0}}\big\{\e^{(0)}_\beta(f+g,f-g)+ k\check{\e}(f+g,f-g)\big\}\nn\\
&=\inf_{f\in\cN_{\Omega,1}}\sup_{g\in \cN_{\Omega,0}}\big\{ \e^{(0)}_\beta(f,f)-\e^{(0)}_\beta(g,g)-2k\check{\e}(f,g)\big\}\lb{nvp}\\
&=\inf_{f\in\cN_{\Omega,1}}\sup_{g\in \cN_{\Omega,0}}\big\{ \e^{(0)}_\beta(f,f)-\e^{(0)}_\beta(g,g)-(-2k)\check{\e}(f,g)\big\}\nn\\
&=\frac{\beta}{\int_\Omega (1-\bE_x[\exp(-\beta\tau_\Omega^{(-k)})])\d x}.\nn
\end{align}
Here we did a change of variable $g\rightarrow -g$ in the third equality. Hence, this equality coupled with the fact that $\Omega$ is bounded, leads to \eqref{lt_k-k}.

Next fix $\beta>0$. For $k\in [0,k_0]$ and $f\in\cN_{\Omega,1}$, we claim that the supremum in \eqref{nvp} is attained by $g\in \cN_{\Omega,0}$ satisfying $\check{\e}(f,g)\leq 0$. To see this, note that for $g\in \cN_{\Omega,0}$ with $\check{\e}(f,g)\geq 0$,
$$
\e^{(0)}_\beta(f,f)-\e^{(0)}_\beta(g,g)-2k\check{\e}(f,g)\leq  \e^{(0)}_\beta(f,f)-\e^{(0)}_\beta(-g,-g)-2k\check{\e}(f,-g).
$$
Combining this claim with the fact that
$
k\rightarrow -2k\check{\e}(f,g)\ \text{is non-decreasing in }[0,k_0]\ \text{for }g\in\cN_{\Omega,0}\ \text{with }\check{\e}(f,g)\leq 0,
$
we see that $\int_\Omega\bE_x[\exp(-\beta\tau_\Omega^{(k)})]\d x$ is non-decreasing for $k\in [0,k_0]$.

Applying similar arguments to \eqref{vp-m}, we obtain the assertions for the mean exit time.
%%%I change the following sentence to above: For the mean exit time, using the fact that
%$$
%\lim_{\beta\rightarrow 0}\frac{1-\exp(-\beta\tau_\Omega^{(k)})}{\beta}=\tau_\Omega^{(k)}
%$$
%gives us the desired results.
\end{proof}

In the following, we consider another family of operators defined by
\be\lb{L-kap-ep}
\cL_{\kappa,\epsilon}= \kappa\nabla\cdot a\nabla+\epsilon\Delta^{\alpha/2},\q \kappa,\epsilon>0.
\de
It can be checked that $\cL_{\kappa,\epsilon}$ is self-adjoint in $L^2(\bR^d,\d x)$. Under (B1), we have a family of processes $X^{(\kappa,\epsilon)}$ with generator $\cL_{\kappa,\epsilon}$ for $\kappa,\epsilon>0$.
For any open set $\Omega\subset\bR^d$, denote  $\tau^{(\kappa,\epsilon)}_\Omega$ by the exit time of process $X^{(\kappa,\epsilon)}$.

\medskip

Thanks to the variational formulas for the exit time again, a comparison theorem for processes $X^{(\kappa,\epsilon)}$, $\kappa,\epsilon>0$ is presented as follows.

\begin{thm}\lb{comp-thm2}
Assume that $\cL_{\kappa,\epsilon}$, $\kappa,\epsilon>0$ defined in \eqref{L-kap-ep} satisfy {\rm(B1)}. Let $\Omega\subset \bR^d$ be a bounded open set. Then for any $\beta>0$,
$\int_\Omega\bE_x[\exp(-\beta\tau_\Omega^{(\kappa,\epsilon)})]\d x$ is non-decreasing and $\int_\Omega\bE_x[\tau_\Omega^{(\kappa,\epsilon)}]\d x$ is non-increasing for both $\kappa$ and $\epsilon$.
%Additionally, if $\kappa=0$, then we have
%\be\lb{order}
%\int_\Omega\bE_x\tau_\Omega^{(0,\epsilon)}\d x= O(\epsilon^{-\alpha}),\q \text{as }\epsilon\rightarrow 0^+.
%\de
\end{thm}
\begin{proof}
For fixed $\kappa,\epsilon>0$, since $\cL_{\kappa,\epsilon}$ is self-adjoint in $L^2(\bR^d,\d x)$, the associated Dirichlet form
\begin{align*}
\e^{(\kappa,\epsilon)}( f,g)&:=\kappa\int_{\bR^d} \nabla f(x)\cdot a(x)\nabla g(x)\d x\\
&\q +\eps\int\int_{x\neq y}\big(f(x)-f(y)\big)\big(g(x)-g(y)\big)\frac{ c_{d,\alpha}}{|x-y|^{d+\alpha}}\d x\d y
%,\q \text{for }f,g\in C^\infty_c(\bR^d).\nn
\end{align*}
is symmetric. So it follows from Corollary \ref{revm-lt} that
$$
\frac{\beta}{\int_\Omega(1-\bE_x[\exp(-\beta\tau_\Omega^{(\kappa,\epsilon)})])\d x}=\inf_{f\in\cN_{\Omega,1}}{ \e_{\beta}^{(\kappa,\epsilon)}}(f,f)\q \text{for }\beta>0,
$$
where $\e_{\beta}^{(\kappa,\epsilon)}(\cdot, \cdot):=\e^{(\kappa,\epsilon)}( \cdot, \cdot)+{\beta}\langle \cdot, \cdot\rangle_2$,
and
$$
\Big(\int_\Omega\bE_x[\tau_\Omega^{(\kappa,\epsilon)}]\d x\Big)^{-1}=\inf_{f\in\cN_{\Omega,1}}\e^{(\kappa,\epsilon)}(f,f).
$$
Therefore, the monotonicity of the Laplace transform of the exit time and the mean exit time follows by the definition of $\e^{(\kappa, \eps)}$ directly.
%If $\kappa=0$, then use Theorem \ref{main-lt} again we get
%$$
%\big(\int_\Omega\bE_x\tau_\Omega^{(0,\epsilon)}\d x\big)^{-1}=\epsilon^{\alpha}\inf_{f\in\cN_{\Omega,1}}\int_{\bR^d}f(x)(-\Delta^{\alpha/2})f(x)\d x.
%$$
%Thus the desired result holds.
\end{proof}

%\begin{rem}
%Note that \cite{IP06,YD08} consider more general stable processes with drifts on $\bR$ and obtain the same order $O(\epsilon^{-\alpha})$ for the mean exit time. We extend part of their results to the higher dimensions in \eqref{order}.
%\end{rem}

%%%%%%%%%%%%%%%%%%%%%%%%%%%%%%%%%%%%%%%%%%%%%%%%%%%%%%%%%%%%%%%%%%%%

\section{ The exit time of reversible ergodic Markov processes}\lb{main-s2}

We consider reversible ergodic Markov processes in this section. In Subsection \ref{proof-mr2} we give a proof of Theorem \ref{main-exp}, while Subsection \ref{exp-poinc} is devoted to the proof of Corollary \ref{up-low-b} and some interesting estimates of the exit time. In final part of this section, we introduce an example.

\subsection{Proof of Theorem \ref{main-exp}}\lb{proof-mr2}

 To prove Theorem \ref{main-exp}, we first begin with some preparation. Recall that  $Y=(Y_t)_{t\geq0}$ is a  right-continuous reversible ergodic strong Markov process on the Polish space $(E,\mathscr{B}(E))$ with transition kernel $P_t(x,dy),t\geq0,x,y\in E$ and stationary distribution $\pi$.
%We use the same notation $\bE_x$ for the expectation starting from $x$ corresponding to process $Y$ and write $\bE_\pi=\int_E\bE_x\pi(\d x)$.
Let $L^2(E,\pi)$ be the space of square integrable functions with usual scalar product $\lan \cdot,\cdot\ran_{2,\pi}$ and norm $||\cdot||_{2,\pi}$. Let $(L,\scr{D}(L))$ be the infinitesimal generator of the process $Y$ in $L^2(E,\pi)$.
The associated Dirichlet form
$\scr{E}_\pi$ is defined by the completion of the bilinear form
	$$ (f, g) \mapsto\langle-L f, g\rangle_{2,\pi},\quad f,g\in \scr{D}(L)$$
	 with respect to the norm $\left(\langle-L\cdot,\cdot\rangle_{2,\pi}+\|\cdot\|_{2,\pi}\right)^{1/2}$,
%% $(\e_\pi,\scr{D}(\e_\pi))$ as $$\e_\pi(f,g)=\lan (-L)f,g\ran_{2,\pi}\q
\text{and} $$\q \scr{D}(\e_\pi):=\{f\in L^2(E,\pi):\e_\pi(f,f)<\infty\}.$$
For any $\beta\in\bR$, define
$$
\e_{\pi,\beta}(f,g)=\e_\pi(f,g)+\beta\lan f,g\ran_{2,\pi},\q f,g\in \scr{D}(\e_\pi).
$$
It is obvious that $\e_{\pi,0}(\cdot,\cdot)=\e_\pi(\cdot,\cdot)$. Throughout this section,
we assume that
$(\e_\pi,\scr{D}(\e_\pi))$ is regular.

For any open set $\Omega\subset E$, we also denote $\tau_\Omega$ the exit time of process $Y$ from $\Omega$. Recall that $\lambda_0(\Omega)$ and $\lambda_1$, defined in \eqref{di-ei} and \eqref{pi-c}, are the optimal constants of local Poincar\'e inequality on $\Omega$ and Poincar\'e inequality respectively.

Now fix an open set $\Omega\subset E$ and $\beta\in\bR$. Consider the following Poisson's equation
\begin{equation}\lb{rev-p}
\begin{cases}
(-\beta-L)u=1,&\q \text{in } \Omega;\\
 u=0,&\q \text{q.e. on }\Omega^c.
\end{cases}
\end{equation}
A function $u_\beta\in \scr{D}(\e_\pi)$ with $u_\beta=0$ q.e. on $\Omega^c$ is called
%%%LJ I add "with $u_\beta=0$ q.e. on $\Omega^c$"
the weak solution of \eqref{rev-p} if
\begin{equation}\lb{weaksol}
\e_{\pi,-\beta}(u_\beta,f)=\pi(f) \q \text{for all } f\in \scr{D}(\e_\pi)\ \text{with }\check{f}=0\ \text{q.e. on }\Omega^c.
\end{equation}
%where $C_0(\Omega)$  is the set of continuous functions which are equal to 0 on $\Omega^c.$

%One can combine Lemma\ref{int-exp} with \cite[Theorem 2.1 and Lemma 2.1]{KAM11} to obtain the  probabilistic representation for \eqref{rev-p} as follows.

By  a modification of the proof of \cite[Theorem 2.1]{KAM11}, we have the following lemma.

\begin{lem}\lb{int-exp}

Let $\Omega\subset E$ be an open set and assume that $\lambda_0(\Omega)>0$.
Then for any $\beta\in (-\infty, 0)\cup (0, \lambda_0(\Omega))$, $\big(\bE_x[\exp(\beta\tau_\Omega)]:x\in E\big)$ is in $\scr{D}(\e_\pi)$. Furthermore,
$v_\beta:=\big((\bE_x[\exp(\beta\tau_{\Omega})]-1)/\beta: x\in E\big)$ is a weak solution of  \eqref{rev-p}.
\end{lem}
\begin{proof}
For $\beta<0$, the desired result is obvious from \cite[Theorem 2.1]{KAM11}. So it suffices to consider the case $\beta\in (0,\lambda_0(\Omega))$ here.

Let $\chi\in C^3(\bR)$ such that $\chi\geq0$, $\chi'\leq0$, $\chi(s)=1,s\leq0$ and $\chi(s)=0,s\geq1$. Let
 $ \theta \in C^3(\mathbb{R}_+)$ satisfy
  $\theta' \geq 0$ and
$\theta(s) = 1, s\geq 1.$
Define
$$
\rho_t(s)=\theta(ts)\int_0^s \beta \exp(\beta r)\chi(r-t)dr,\q s\geq0,\ t\geq1,
$$
and
$h_{\rho_t}(x)=\bE_x[\rho_t(\tau_{\Omega})]$ for $x\in E$.
By \eqref{poinc-l} and \cite[Corollary 2.1 and (10)]{KAM11}, there exists a constant $C>0$ such that
$$\pi(h_{\rho_t}^2)\leq \frac{1}{\lambda_0(\Omega)}\e_\pi(h_{\rho_t},h_{\rho_t})=\frac{1}{\lambda_0(\Omega)}\lan h_{\rho'_t},h_{\rho_t}\ran_\pi\leq \frac{\beta}{\lambda_0(\Omega)}\pi(h_{\rho_t}^2)+\frac{C}{\lambda_0(\Omega)}\pi(h_{\rho_t}),$$
which implies that
$$\left(1-\frac{\beta }{\lambda_0(\Omega)}\right)\pi(h_{\rho_t}^2)\leq \frac{ C}{\lambda_0(\Omega)}\pi(h_{\rho_t}).$$
Moreover, it follows by the Jensen's inequality that  $\pi(h_{\rho_t})\leq ||h_{\rho_t}||_{2,\pi}$.
Thus combining the above inequalities with the setting $0<\beta<\lambda_0(\Omega)$, we have that
$$
||h_{\rho_t}||_{2,\pi}\leq \frac{C}{\lambda_0(\Omega)-\beta}<\infty.
$$
That is, the $L^2$-norms of functions $h_{\rho_t},t\geq1$ are uniformly bounded. This, together with the fact that $\rho_t$ is increasing to function $\rho(s)=\exp(\beta s)-1$ as $t\rightarrow\infty$, yields that
$$
h_\rho(x):=\bE_x[\exp(\beta\tau_\Omega)]-1,\q x\in E
$$
belongs to $L^2(E,\pi)$, and $h_{\rho_t}\rightarrow h_\rho$ as $t\rightarrow\infty$ in $L^2(E,\pi)$. Furthermore, from the similar argument in \cite[p.77]{KAM11}, we see that
 $h_{\rho_t}$ converges weakly to $h_\rho$ with respect to $\e_{\pi,1}$ as $t\rightarrow \infty$. Hence $h_\rho\in \mathscr{D}(\e_\pi)$.

Next for any $f\in \mathscr{D}(\e_\pi)$ with $\check{f}=0$ q.e. on $\Omega^c$, since $(\e_\pi,\scr{D}(\e_\pi))$ is regular, there exists a sequence of functions $(f_n)_{n\geq1}\subset C_0(\Omega)\cap\scr{D}(\e_\pi)$ such that $f_n\rightarrow f$ with respect to $\e_{\pi,1}$ as $n\rightarrow \infty$.
Thus combining this fact with the above analysis and the argument in \cite[p.77]{KAM11} gives the desired result.
\end{proof}

\begin{rem}\lb{p-lp}
In \cite[Theorem 2.1]{KAM11}, the authors showed that
$$
\big(\bE_x [\exp(\beta\tau_\Omega)],x\in E\big)\in \scr{D}(\e_\pi)\q \text{for all } 0<\beta<\lambda_1\pi(\Omega^c)
$$
if $\pi(\Omega^c)>0$. Since by the proof of \cite[Theorem 3.1]{cmf00} or \cite[Theorem 4.10]{cmf05},
\be\lb{poin-lpoin}
\lambda_0(\Omega)\geq \lambda_1\pi(\Omega^c)\q \text{if}\q \pi(\Omega^c)>0,
\de
Lemma \ref{int-exp} provides a more general result.
\end{rem}

With the above preparation in hand, we are ready to prove Theorem \ref{main-exp}.

\medskip
\noindent{\bf Proof of Theorem \ref{main-exp}.} 	
 Let us first consider the case that $\beta\in(0,\lambda_0(\Omega))$ if $\lambda_0(\Omega)>0$. According to Lemma \ref{int-exp}, $v_\beta:=\big((\bE_x[\exp(\beta\tau_{\Omega})]-1)/\beta:x\in E  \big)$ is a weak solution of \eqref{rev-p}. Further, $\bar{v}_\beta:=v_\beta/\pi(v_\beta)\in \cN^\pi_{\Omega,1}$.
Thus for any $f\in\mathscr{D}(\e_\pi)$ with $\check{f}=0$ q.e. on $\Omega^c$,
\be\lb{jc}
\e_{\pi,-\beta}(\bar{v}_\beta,f)=\frac{1}{\pi(v_\beta)}\pi(f).
\de
In particular,
	\begin{equation}\lb{epn-v}
	\e_{\pi,-\beta}(\bar{v}_\beta,\bar{v}_\beta)=\frac{1}{ \pi(v_\beta)}=\frac{\beta}{ \bE_\pi[\exp(\beta\tau_{\Omega})]-1}.
	\end{equation}

For any $f\in \cN^\pi_{\Omega,1},$ let $f_1=f-\bar{v}_\beta.$ Then $f_1\in \scr{D}(\e_\pi)$ with $\check{f}_1=0$ q.e. on $\Omega^c$ and $\pi(f_1)=0$ by $\bar{v}_\beta\in \cN^\pi_{\Omega,1}$. Since $L$ is self-adjoint with respect to $\pi$, we combine with \eqref{jc} to obtain
	\begin{equation}\label{eq lower}
	\begin{split}
	\e_{\pi,-\beta}(f,f)&=\e_{\pi,-\beta}(f_1+\bar{v}_\beta,f_1+\bar{v}_\beta)=\e_{\pi,-\beta}(f_1,f_1)+\e_{\pi,-\beta}(\bar{v}_\beta,\bar{v}_\beta)\\
	&\geq \e_{\pi,-\beta}(\bar{v}_\beta,\bar{v}_\beta),
	\end{split}
	\end{equation}
	where in the last inequality we used the fact that
	$$
	\e_{\pi,-\beta}(f_1,f_1)=\e_\pi(f_1,f_1)-\beta\pi(f_1^2)\geq (\lambda_0(\Omega)-\beta)\pi(f_1^2)\geq 0
	%,\q \text{for all }f_1\in \scr{D}(\e_\pi)
	$$
	for $0<\beta<\lambda_0(\Omega)$.
	So the desired assertion follows by \eqref{epn-v}--\eqref{eq lower}.

 We now consider the case that $\beta\ge\lambda_0(\Omega)$. It is known that $\bE_\pi[\exp(\beta\tau_\Omega)]=\infty$ for $\beta\geq\lambda_0(\Omega)$ (see e.g. \cite{Fr73,LLL14}). So in this case, the left-hand side of \eqref{vp-exp} is equal to zero.
 Moreover, by the definition of $\lambda_0(\Omega)$,  one could see that for any $\epsilon>0$, there exists function $f_\epsilon\in\mathscr{D}(\e_\pi)$ with $\check{f}_\epsilon=0$ q.e. on $\Omega^c$ and $\pi(f^2_\epsilon)=1$ such that
$
\e_\pi(f_\epsilon,f_\epsilon)\leq (\lambda_0(\Omega)+\epsilon).
$
Therefore, for $\beta>\lambda_0(\Omega)$, by taking $\epsilon<\beta-\lambda_0(\Omega)$, we have that
$$
\e_{\pi,-\beta}(\frac{f_\epsilon}{\pi(f_\epsilon)},\frac{f_\epsilon}{\pi(f_\epsilon)})\leq (\lambda_0(\Omega)-\beta+\epsilon)\frac{1}{\pi(f_\epsilon)^2}<0.
$$
Thus the right-hand side of \eqref{vp-exp} is also zero when $\beta>\lambda_0(\Omega)$.
Furthermore, since $\beta\rightarrow\e_{\pi,-\beta}(f,f)$ is continuous, our assertion holds true for $\beta=\lambda_0(\Omega)$.
\qed

\medskip

Under the ergodicity, for fixed open set $\Omega\subset E$ with $\pi(\Omega^c)>0$, we see that for $\pi$-a.e. $x$, $\bE_x \tau_\Omega<\infty$  by a similar argument in the proof of \cite[Lemma 2.1]{Mao02}. This implies that the Dirichlet form $(\e_\pi^\Omega,\mathscr{D}(\e_\pi^\Omega))$, defined by $\e_\pi^\Omega=\e_\pi$ on $\mathscr{D}(\e_\pi^\Omega):=\{f\in\mathscr{D}(\e_\pi):\ \check{f}=0\ \text{q.e. on }\Omega^c\}$, is transient. Thus similar to Theorem \ref{main-lt}, we obtain the following variational formulas for the Laplace transform of the exit time and the mean exit time starting from $\pi$. The proof, which we omit, is almost same as that of Theorem \ref{main-lt}.

\begin{thm}\lb{rev-v1}
Let $\Omega\subset E$ be an open set. Then for any $\beta>0$,
$$\frac{\beta}{1-\bE_\pi[\exp(-\beta\tau_{\Omega})]}=\inf_{f\in\cN^\pi_{\Omega,1}}\e_{\pi,\beta}(f,f). $$
Additionally, if $\pi(\Omega^c)>0$, then
$$
1/\bE_\pi[\tau_\Omega]=\inf_{f\in\cN^\pi_{\Omega,1}}\e_\pi(f,f).
$$	
\end{thm}

%%%%%%%%%%%%%%%%%%%%%%%%%%%%%%%%%%%%%%%%%%%%%%%%%%%%%%%%
\subsection{Bounds for the exit time}\lb{exp-poinc}

In the literature, the relation of the exponential moments of the exit (hitting) time and  (local) Poincar\'e inequality is an important topic in ergodic theory (see e.g. \cite{cmf04} for Markov chains and \cite{KAM11} for symmetric Markov processes). As we have seen in Lemma \ref{int-exp}, some qualitative analysis of them is provided. Indeed, from Theorem \ref{main-exp}, we could prove Corollary \ref{up-low-b} which presents some quantitative inequalities between the exponential moments of the exit time, $\lambda_0(\Omega)$ and $\lambda_1$ as follows.

\medskip
\noindent
{\bf Proof of Corollary \ref{up-low-b}.}
We consider the upper bound first.
For any function $f\in \cN^\pi_{\Omega,1}$,
let $g=f/||f||_{2,\pi},$ then $\pi(g^2)=1$ and $\check{g}=0$ q.e. on $\Omega^c$.
So by the definition of $\lambda_0(\Omega)$, we arrive at $\e_\pi(g,g)\geq \lambda_0(\Omega)$, i.e., $\e_\pi(f,f)\geq \lambda_0(\Omega)\pi(f^2)$. Hence, for any $0<\beta<\lambda_0(\Omega)$,
$$
\inf_{f\in \cN^\pi_{\Omega,1}}\e_{\pi,-\beta}(f,f)\geq (\lambda_0(\Omega)-\beta)\inf_{f\in \cN^\pi_{\Omega,1}}\pi(f^2)\geq \lambda_0(\Omega)-\beta,
$$
where in the second inequality we used the fact $\pi(f^2)\geq \pi(f)^2=1$ for $f\in \cN^\pi_{\Omega,1}$. Thus \eqref{rep-expint} follows immediately from Theorem \ref{main-exp}.
Additionally if  $\lambda_1>0$
and $\pi(\Omega^c)>0$, \eqref{exp-lambda1} is obtained by \eqref{rep-expint} and \eqref{poin-lpoin} together.

Next we turn to the lower bound. By the definition of $\phi$, we have
\be \lb{diri-phi}
\e_\pi(\phi,\phi)=\lambda_0(\Omega)\pi(\phi^2).
\de
It is obvious that $\eqref{expint-lb}$ holds if $\pi(\phi)=0$, thus it suffices to consider the case that $\pi(\phi)\neq0$. Indeed, since $\phi/\pi(\phi)\in \cN^\pi_{\Omega,1}$,
from Theorem \ref{main-exp} we have
$$
\frac{\beta}{\bE_\pi[\exp(\beta\tau_\Omega)]-1}\leq  \frac{\e_{\pi,-\beta}(\phi,\phi)}{\pi(\phi)^2}= \frac{(\lambda_0(\Omega)-\beta)\pi(\phi^2)}{\pi(\phi)^2}
$$
for $0<\beta<\lambda_0(\Omega)$, which yields \eqref{expint-lb}.
\qed
\medskip

\begin{rem}
\begin{itemize}
\item[(1)]
We note that according to \cite[Theorem 2.2]{LLL14} with $f= 1$ and $r(t)=\exp(\beta t)$, it also has
$$
\frac{\mathbb{E}_\pi[\exp(\beta\tau_{\Omega})]-1}{\beta}=\int_{0}^{\infty}\exp(\beta t)\|P_{t/2}^\Omega\|^2_{2,\pi}\d t\leqslant \frac{1}{\lambda_0(\Omega)-\beta},
$$
where $P_{t}^\Omega$ is the semigroup of process $Y$ killed upon leaving $\Omega$.
In Corollary \ref{up-low-b}, we obtained the same upper bounds by a new proof.

\item[(2)] Since  $\bE_\pi[\exp(-\beta\tau_\Omega)]\geq1/\bE_\pi[\exp(\beta\tau_\Omega)]$ for all $\beta>0$, \eqref{rep-expint} yields that
$$
\bE_\pi[\exp(-\beta\tau_\Omega)]\geq 1-\frac{\beta}{\lambda_0(\Omega)},\q \text{for }0<\beta<\lambda_0(\Omega).
$$
In Corollary \ref{hit time} below, we will provide a more precise lower bound for the Laplace transform of the exit time by its variational formula directly.
\end{itemize}
\end{rem}

\begin{cor}\label{hit time}
Let $\Omega\subset E$ be an open set  and assume that $\lambda_0(\Omega)>0$.
Then the following statements hold.
\begin{itemize}
\item[\rm(1)] For all $\beta>0$,
	$$
	\bE_\pi[\exp(-\beta\tau_\Omega)]\geq 1-\frac{\beta}{\lambda_0(\Omega)+\beta}\q \text{and}\q \bE_\pi[\tau_\Omega]\leq 1/\lambda_0(\Omega).
   $$

\item[\rm(2)] If there exists a weak solution $\phi$ of \eqref{dirichlet-e},
then for all $\beta>0$,
\be\lb{lt-ub}
\bE_\pi[\exp(-\beta\tau_\Omega)]\leq 1-\frac{\beta\pi(\phi)^2}{(\lambda_0(\Omega)+\beta)\pi(\phi^2)} \q\text{and}\q  \bE_\pi[\tau_\Omega]\geq\frac{ \pi(\phi)^2}{\lambda_0(\Omega)\pi(\phi^2)}.
\de
\end{itemize}
\end{cor}
\begin{proof}
(1) Applying the similar arguments in the proof of Corollary \ref{up-low-b} to Theorem \ref{rev-v1},
 we get that
\begin{equation}\label{vari-exp}
\frac{\beta}{1-\bE_\pi[\exp(-\beta\tau_{\Omega})]}=\inf_{f\in\cN^\pi_{\Omega,1}}\left[\left(\frac{\e_{\pi}(f,f)}{\pi(f^2)}+\beta\right)\pi(f^2)\right]\geq\lambda_0(\Omega)+\beta.
\end{equation}

We turn to the mean exit time.
Observe that
$$
\frac{1-\exp(-\beta\tau_{\Omega})}{\beta}\leq\tau_{\Omega}\q \text{ and }\q \lim_{\beta\rightarrow0}\frac{1-\exp(-\beta\tau_{\Omega})}{\beta}=\tau_{\Omega},
$$
thus it follows from the dominated convergence theorem that
\be\lb{expint-E}
\bE_\pi[\tau_{\Omega}]=\lim_{\beta\rightarrow0}\frac{1-\bE_\pi[\exp(-\beta\tau_{\Omega})]}{\beta}.
\de
%Since $$\e_{\pi,-\lambda}(\frac{1-v_{-\lambda}}{\lambda},f)=\pi(f),$$
%then by dominated convergence theorem, $w$ solves \eqref{rev-p} with $\xi=-1$ and $\lambda=0.$
Hence, by letting $\beta\rightarrow0$ in \eqref{vari-exp}, we obtain that
$\bE_\pi[\tau_{\Omega}]\leq1/{\lambda_0(\Omega)}.$
%Furthermore, $\bE_\pi\tau_{\Omega}\leq {1}/{\lambda_1\pi(\Omega^c)}$ by using the fact that $\lambda_0(\Omega)\geq\lambda_1\pi(\Omega^c)$ again.

(2) Suppose that there exists a weak solution $\phi$ of \eqref{dirichlet-e} so that  \eqref{diri-phi} holds.
We only need to consider the case that $\pi(\phi)\neq0$ since \eqref{lt-ub} is obvious if $\pi(\phi)=0$. Combining Theorem \ref{rev-v1} with $\phi/\pi(\phi)\in \cN^\pi_{\Omega,1}$ and \eqref{diri-phi},
$$
\frac{\beta}{1-\bE_\pi[\exp(-\beta\tau_\Omega)]}\leq \frac{\e_{\pi,\beta}(\phi,\phi)}{\pi(\phi)^2}=\frac{(\lambda_0(\Omega)+\beta)\pi(\phi^2)}{\pi(\phi)^2}
$$
for all $\beta>0$. By similar arguments in (1) for the mean exit time, the proof is complete.
\end{proof}

\begin{rem}
\begin{itemize}
\item[(1)] If $\lambda_1>0$
and $\pi(\Omega^c)>0$, then Corollary \ref{hit time}, together with \eqref{poin-lpoin} implies that
$$
	\bE_\pi[\exp(-\beta\tau_\Omega)]\geq 1-\frac{\beta}{\lambda_1\pi(\Omega^c)+\beta}\q \text{and}\q \bE_\pi[\tau_{\Omega}]\leq\frac{1}{\lambda_1\pi(\Omega^c)}.
$$

\item[(2)]  When $\lambda_0(\Omega)>1$, by combining Corollaries \ref{up-low-b} and \ref{hit time}, we could obtain an interesting inequality
$$
\sum_{n=0}^\infty\frac{\bE_\pi[\tau^{2n+1}_\Omega]}{(2n+1)!}\leq \frac{\lambda_0(\Omega)}{(\lambda_0(\Omega)-1)(\lambda_0(\Omega)+1)},
$$
since
$
2\sum_{n=0}^\infty\frac{\bE_\pi[\tau^{2n+1}_\Omega]}{(2n+1)!}=\bE_\pi[\exp(\tau_\Omega)]-\bE_\pi[\exp(-\tau_\Omega)].
$
\end{itemize}
\end{rem}
\medskip

Next, from Corollaries \ref{up-low-b} and \ref{hit time}, we will present some estimates of the exit time by Lyapunov conditions.

\begin{cor}\lb{lyap}
Under the conditions given in Corollary \ref{up-low-b} and assume that there exists a Lyapunov function $\varphi$ which is locally bounded below satisfying $\varphi= 0$  on $\Omega^c$, $\varphi|_{\Omega}> 0$ and
$$
\delta:=-\sup_{\Omega}\frac{L\varphi}{\varphi}>0.
$$
We also assume that the eigenfunction of $L$ corresponding to $\lambda_0(\Omega)$ is locally bounded above. Then

$$
\aligned
&\bE_\pi[\tau_{\Omega}]\leq  1/\delta,\q \bE_\pi[\exp(\beta\tau_{\Omega})]\leq  1+\frac{\beta}{\delta-\beta}\q \text{for }0<\beta<\delta,\\
\endaligned
$$
and
$$
\bE_\pi[\exp(-\beta\tau_{\Omega})]\geq 1-\frac{\beta}{\delta+\beta} \q \text{for }\beta>0.
$$
\end{cor}
\begin{proof}
 Note that by \cite[Theorem 3.2]{cmf00}, $\lambda_0(\Omega)\geq \delta$, thus the conclusion is obtained by Corollaries \ref{up-low-b} and  \ref{hit time}.
\end{proof}

%%%%%%%%%%%%%%%%%%%%%%%%%%%%%%%%%%%%%%%%%%%%%%%%%%%%%%%%%%%%%%%%%%%%%%%
%%%%%%%%%%%%%%%%%%%%%%%%%%%%%%%%%%%%%%%%%%%%%%%%%%%%%%%%%%%%%%%%%%%%%%%
\subsection{An example of reversible diffusion processes on manifold}\lb{exm-sym}
%In this subsection, we give a concrete example to illustrate the applications of the main results presented in the previous subsections.

Let $M$ be a $d$-dimensional connected complete Riemannian manifold  with Riemannian metric $\rho$ and Riemannian volume $\d x$.
Consider an operator $L=\Delta+ \nabla V\cdot\nabla$
where $V\in C^2(M)$, and let $Y$ be an ergodic diffusion process on $M$ with the generator $L$ and the stationary distribution $\pi(\d x) =\exp{(V (x))}\d x/\int_M \exp{(V (x))}\d x$.
For fixed $o\in M$, let $\rho(x)$ be the Riemannian distance function from $o$, and let $\mathrm{cut}(o)$ be the cut-locus. Assume that either $\partial M$ is bounded or $M$ is convex.
In this case, the associated Dirichlet form is given by
$$
\e_{\pi}(f,g)=\int_M\nabla f\cdot\nabla g\d\pi\quad \text{and}\quad \scr{D}(\e_\pi)=C^1(M)\cap L^2(M,\pi).
$$

Applying Corollaries \ref{up-low-b} and \ref{hit time} with the variational formula of $\lambda_0(\cdot)$ in \cite[Theorem 1.2]{wfy99}, some estimates of the exit time are obtained as follows.

\begin{cor}
Fix $r>0$ and let
$\Omega=B_r^c:=\{x\in M:\rho(x)>r\}$
%% $\Omega=B_r^c\subset M$
  with $\pi(\Omega^c)>0$ and $\partial \Omega\neq \emptyset$.
%and the first Dirichlet eigenvalue $\lambda_0(\Omega)>0$.
Denote
$$\gamma(r)=\sup_{\rho(x)=r,x\notin \mathrm{cut}(o)}L\rho(x)\q\text{and}\q C(r)=\int_{1}^{r}\gamma(s)\d s.$$
Assume that
\begin{align}\label{d:del}
\delta_r:=\sup_{t\geq r}\int_{r}^{t}\exp{(-C(l))\d l}\int^{\infty}_{ t}\exp{(C(s))}\d s<\infty.
\end{align}
Then $ \bE_\pi[\tau_{\Omega}]\leq 4\delta_r$,
\begin{align*}
\bE_\pi[\exp(\beta\tau_{\Omega})]\le1+\frac{4\beta\delta_r}{1-4\beta\delta_r}\q \text{for }0<\beta< 1/(4\delta_r)
\end{align*}
and
$$
\bE_\pi[\exp(-\beta\tau_{\Omega})]\geq 1-\frac{4\beta\delta_r}{1+4\beta\delta_r}\q \text{for }\beta>0.
$$
\end{cor}
\begin{proof}

First, by \cite[Theorem 1.2]{wfy99}, for any positive  function $f\in C[r,\infty)$ we have
\be\lb{diff-expinq}
\lambda_0(\Omega)\geq \inf_{t\geq r}f(t)\left(\int_{r}^{t}\exp{(-C(l))\d l}\int_{l}^{\infty}\exp{(C(s))}f(s)\d s\right)^{-1}.
\de
Denote $\varphi(t)=\int_{r}^{t}\exp{(-C(l))}\d l$ and take $f(t)=\sqrt{\varphi(t)}$ in \eqref{diff-expinq}.
From the integration by parts, for $l\geq r$ we see that
\begin{equation*}
\begin{split}
\int_{l}^{\infty}\exp{(C(s))}f(s)\d s &=-\int_{l}^{\infty}\sqrt{\varphi(s)}\d \left(\int_{s}^{\infty}\exp{(C(t))}\d t\right)\\
&\leq  \frac{\delta_r}{\sqrt{\varphi(l)}}+\frac{\delta_r}{2}\int_{l}^{\infty}\frac{\varphi'(s)}{\varphi^{3/2}(s)}\d s,
\end{split}
\end{equation*}
which implies that
%%%LJ I change "That is," to above
 $$\int_{l}^{\infty}\exp{(C(s))}f(s)\d s \leq\frac{2\delta_r}{\sqrt{\varphi(l)}}\q \text{for all }l\geq r.$$
 Hence,
 $$\int_{r}^{t}\exp{(-C(l))\d l}\int^{\infty}_{ l}\exp{(C(s))}f(s)\d s\leq2\delta_r\int_{r}^{t}\frac{\exp{(-C(l))}}{\sqrt{\varphi(l)}}\d l{\le } 4\delta_r\sqrt{\varphi(t)}.$$
 Combining this with \eqref{d:del}--\eqref{diff-expinq}, we obtain
\be\lb{lamb-Br}\lambda_0(\Omega)\geq 1/(4\delta_r)>0.\de
So plugging \eqref{lamb-Br} into \eqref{rep-expint} yields that
$$
\bE_\pi[\exp(\beta\tau_{\Omega})]\le1+\frac{4\beta\delta_r}{1-4\beta\delta_r}\q \text{for }0<\beta< 1/(4\delta_r).
$$
The rest of assertions can be obtained by applying \eqref{lamb-Br} to Corollary \ref{hit time}.
\end{proof}

\section*{Appendix. A simple property of ``q.e.''}
Let $\Omega\subset E$ be an open set, and assume that $f=0$ and $g=0$ q.e. on $\Omega^c$. In the following, we claim that $f\pm g=0$ q.e. on $\Omega^c$.

Indeed, since $f=0$ and $g=0$ q.e. on $\Omega^c$, there exist $N_1, N_2\subset \Omega^c$, with  $\mathrm{Cap}^{\alpha}(N_1)=\mathrm{Cap}^{\alpha}(N_2)=0$, such that
$f=0$ on $\Omega^c\setminus N_1$ and
$g=0$ on $\Omega^c\setminus N_2$. Hence, we obtain that
$$
f\pm g=0\quad \text{on}\quad \Omega^c\setminus( N_1\cup N_2).
$$
We now prove that $\mathrm{Cap}^{\alpha}(N_1\cup N_2)=0$, therefore the claim holds. By the definition of capacity (see e.g. \cite[(2.1.9)]{O13}), there exist sequences of open sets $(D_{i,n})_{n\geq1}$, $i=1,2$ , such that $N_i\subset D_{i,n}$ and $$\mathrm{Cap}^{\alpha}(D_{i,n})\downarrow \mathrm{Cap}^{\alpha}(N_i).$$
Set $D_n=D_{1,n}\cup D_{2,n}$, then $ N_1\cup N_2\subset D_n$. It follows from \cite[Lemma 2.1.2]{O13} that
$$\mathrm{Cap}^{\alpha}(D_{n})\leqslant\mathrm{Cap}^{\alpha}(D_{1,n})+\mathrm{Cap}^{\alpha}(D_{2,n}),$$
which implies that $\mathrm{Cap}^{\alpha}(N_1\cup N_2)=0$ by the definition of capacity again.

\vskip 0.3truein
%----------------------------------------------------------------------------------------------
%----------------------------------------------------------------------------------------------
{\bf Acknowledgement}\
The authors would like to thank the reviewers for corrections and helpful comments; and to Professors C.-R. Hwang and J. Wang for their suggestions which permitted to improve the presentation of the article.
Lu-Jing Huang acknowledges support from NSFC (No. 11901096),
NSF-Fujian(No. 2020J05036), the Program for Probability and Statistics: Theory and
Application (No. IRTL1704), and the Program for Innovative Research Team in Science
and Technology in Fujian Province University (IRTSTFJ).
Kyung-Youn Kim thanks Professor C.-R. Hwang for the supervision(MOST 108-2115-M-004-004-), and Institute of Mathematics, Academia Sinica,  Taiwan for the support. 
Yong-Hua Mao and Tao Wang acknowledge support from NSFC (No. 11771047).

%%%%%%%%%%%%%%%%%%%%%%%%%%%%%%%%%%%%%%%%%%%%%%%%%%%%%%%%%%%%%%
%\newpage
%\begin{thebibliography}{00}

%\end{thebibliography}
\bibliographystyle{plain}
\bibliography{jump_drift}

\vskip 0.3truein

{\bf Lu-Jing Huang:}
%\vskip -.1truein
College of Mathematics and Informatics, Fujian Normal University, Fuzhou, 350007, P.R. China. E-mail: \texttt{huanglj@fjnu.edu.cn}

\medskip
{\bf Kyung-Youn Kim:}
%\vskip -.1truein
Department of Applied Mathematics, National Chung Hsing University, Taichung, Taiwan. E-mail: \texttt{kyungyoun07@gmail.com}

\medskip
{\bf Yong-Hua Mao:}
%\vskip -.1truein
Laboratory of Mathematics and Complex Systems(Ministry of Education), School of Mathematical Sciences, Beijing Normal University, Beijing 100875, P.R. China. E-mail: \texttt{maoyh@bnu.edu.cn}

\medskip
{\bf Tao Wang:}
%\vskip -.1truein
Laboratory of Mathematics and Complex Systems(Ministry of Education), School of Mathematical Sciences, Beijing Normal University, Beijing 100875, P.R. China. E-mail: \texttt{wang\_tao@mail.bnu.edu.cn}

\end{document}